\documentclass[a4paper,12pt]{article}
\usepackage{amssymb,amsmath,amsthm,latexsym}
\usepackage{amsfonts}
\usepackage{graphicx,caption,subcaption}
\usepackage[pdftex,bookmarks,colorlinks=false]{hyperref}
\usepackage[hmargin=1.2in,vmargin=1.2in]{geometry}
\usepackage{enumitem}
\usepackage{authblk}
\usepackage{fancyhdr}
\newtheorem{thm}{Theorem}[section]

\newtheorem{defn}[thm]{Definition}

\newtheorem{prob}[thm]{Problem}
\newtheorem{prop}[thm]{Proposition}
\newtheorem{rem}[thm]{Remark}

\title{\sc Topological Integer Additive Set-Graceful Graphs}

\author{\sc N. K.Sudev\footnote{Corresponding author}}
\affil{\small Department of Mathematics,\\ Vidya Academy of Science \& Technology, \\ Thalakkottukara, Thrissur - 680501, Kerala, India,\\ email: {\em sudevnk@gmail.com}}

\author{\sc K. A. Germina}
\affil{\small PG \& Research Department of Mathematics,\\ Mary Matha Arts \& Science College, \\ Mananthavady, Wayanad - 670645, Kerala, India,\\ email: {\em srgerminaka@gmail.com}}

\author{K. P. Chithra}
\affil{\small Naduvath Mana, Nandikkara\\ Thrissur-680301, Kerala, India, \\ email: {\em chithrasudev@gmail.com}}

\date{}

\begin{document}
\maketitle
	
\begin{abstract}
Let $\mathbb{N}_0$ denote the set of all non-negative integers and $X$ be any subset of $X$. Also denote the power set of $X$ by $\mathcal{P}(X)$. An integer additive set-labeling (IASL) of a graph $G$ is an injective function $f:V(G)\to \mathcal{P}(X)$ such that the induced function $f^+:E(G) \to \mathcal{P}(X)$ is defined by $f^+ (uv) = f(u)+ f(v)$, where $f(u)+f(v)$ is the sumset of $f(u)$ and $f(v)$. An IASL $f$ is said to be a topological IASL (Top-IASL) if $f(V(G))\cup \{\emptyset\}$ is a topology of the ground set $X$. An IASL is said to be an integer additive set-graceful labeling (IASGL) if for the induced edge-function $f^+$, $f^+(E(G))= \mathcal{P}(X)-\{\emptyset, \{0\}\}$. In this paper, we study certain types of IASL of a given graph $G$, which is a topological integer additive set-labeling as well as an integer additive set-graceful labeling of $G$. 
\end{abstract}

	
\vspace{0.05in}
	
\section{Introduction}

For all terms and definitions, other than newly defined or specifically mentioned in this paper, we refer to \cite{BM}, \cite{FH} and \cite{DBW}. For different graph classes, we further refer to \cite{BLS} and \cite{JAG}. Unless mentioned otherwise, the graphs considered in this paper are simple, finite, non-trivial and connected.

Let $A$and $B$ be two non-empty sets. The {\em sumset} of $A$ and $B$ is denoted by $A+B$ and is defined by $A+B=\{a+b: a\in A, b\in B\}$. Using the concepts of sumsets of two sets, the following notion has been introduced.

Let $\mathbb{N}_0$ denote the set of all non-negative integers and $\mathcal{P}(\mathbb{N}_0)$ be its power set. An {\em integer additive set-labeling} (IASL) of a graph $G$ is an injective function $f:V(G)\to \mathcal{P}(\mathbb{N}_0)$ such that the induced function $f^+:E(G) \to \mathcal{P}(\mathbb{N}_0)$ is defined by $f^+ (uv) = f(u)+ f(v)$, where $f(u)+f(v)$ is the sumset of $f(u)$ and $f(v)$. A graph that admits an IASL is called an {\em integer additive set-labeled graph} (IASL-graph).

An IASL $f$ of a given graph $G$ is said to be an {\em integer additive set-indexer} (IASI) if the associated function $f^+$ is also injective. 

An IASL (or an IASI) $f$ of a graph $G$ is said to be a $k$-uniform IASL (or a $k$-uniform IASI) if $f^+(uv)=k~\forall~uv\in E(G)$.

The cardinality of the set-label of an element (a vertex or an edge) of a graph $G$ is said to be the {\em set-indexing number} of that element. An element of a graph $G$ is said to be {\em mono-indexed} if its set-indexing number is $1$. 

Since the set-label of every edge of $G$ is the sumset of the set-labels of its end vertices, it can be noted that no vertex of an IASL-graph $G$ can have the empty set as its set-label. If any of the given two sets is countably infinite, then their sumset is also a countably infinite set. Hence, all sets we consider in this paper are non-empty finite sets of non-negative integers.

An {\em integer additive set-graceful labeling} (IASGL) of a graph $G$ is defined in \cite{GS14} as an integer additive set-labeling $f:V(G)\to \mathcal{P}(X)-\{\emptyset\}$ such that the induced function $f^{+}(E(G))=\mathcal{P}(X)-\{\emptyset,\{0\}\}$. A graph $G$ which admits an integer additive set-graceful labeling is called an {\em integer additive set-graceful graph} (in short, IASG-graph).

\noindent The major results on IASG-graphs, established in \cite{GS14}, are the following.

\begin{prop}\label{P-IASGL0a}
{\rm \cite{GS14}} If $f:V(G)\to \mathcal{P}(X)-\{\emptyset\}$ is an integer additive set-graceful labeling on a given graph $G$, then $\{0\}$ must be a set-label of one vertex of $G$.
\end{prop}

\begin{thm}
Let $G$ be an IASG-graph which admits an IASGL $f$ with respect to a finite non-empty set $X$. Then, $G$ must have at least $|X|-1$ pendant vertices.
\end{thm}

\begin{prop}\label{O-IASGL1b}
{\rm \cite{GS14}} Let $G$ be an IASG-graph. Then, there are at least $1+2^{n-1}$ vertices of $G$ adjacent to the vertex having the set-label $\{0\}$, where $n$ is the cardinality of the ground set $X$.
\end{prop}

\begin{prop}\label{P-IASGL2}
{\rm \cite{GS14}} Let $f:V(G)\to \mathcal{P}(X)-\{\emptyset\}$ be an integer additive set-graceful labeling on a given graph $G$ and let $x_n$ be the maximal element of $X$. Then, $x_n$ is an element of the set-label of a vertex $v$ of $G$ if $v$ is a pendant vertex that is adjacent to the vertex labeled by $\{0\}$.
\end{prop}

\begin{prop}\label{P-IASGL2a}
{\rm \cite{GS14}} Let $A_i$ and $A_j$ are two distinct subsets of the ground set $X$ and let $x_i$ and $x_j$ be the maximal elements of $A_i$ and $A_j$ respectively. Then, $A_i$ and $A_j$ are the set-labels of two adjacent vertices of an IASG-graph $G$ is that $x_i+x_j\le x_n$, the maximal element of $X$.
\end{prop}

\begin{thm}\label{T-IASGL3}
{\rm \cite{GS14}} A graph $G$ admits an integer additive set-graceful labeling, then it has even number of edges.  
\end{thm}

\begin{thm}\label{T-IASGL6a}
{\rm \cite{GS14}} Let $X$ be a non-empty finite set of non-negative integers. Then, a graph $G$ admits a graceful IASI if and only if the following conditions hold.
\begin{enumerate}\itemsep0mm
\item[(a)] $0\in X$ and $\{0\}$ be a set-label of some vertex, say $v$, of $G$
\item[(b)] the number of pendant vertices in $G$ is the number of subsets of $X$ which are not the non-trivial summands of any subsets of $X$.
\item[(c)] the minimum degree of the vertex $v$ is equal to the number of subsets of $X$ which are not the sumsets of any two subsets of $X$ or not non-trivial summands of any other subsets of $X$.
\item[(d)] the minimum number of pendant vertices that are adjacent to a given vertex of $G$ is the number of subsets of $X$ which are neither the non-trivial sumsets of any two subsets of $X$ nor the non-trivial summands of any subsets of $X$. 
\end{enumerate}
\end{thm}

\begin{thm}\label{T-IASGL6b}
{\rm \cite{GS14}} A tree $G$ is an IASG-graph if and only if it is a star $K_{1,\, 2^n-2}$, for some positive integer $n$.
\end{thm}

An integer additive set-indexer $f$ of a graph $G$, with respect to a ground set $X\subset \mathbb{N}_0$, is said to be a {\em topological IASL} (Top-IASL) of $G$ if $\mathcal{T}=f(V(G))\cup \{\emptyset\}$ is a topology on $X$. Certain characteristics and structural properties of Top-IASL-graphs have been studied in \cite{GS13}.

The following are the major results on topological IASL-graphs made in \cite{GS13}.

\begin{prop}\label{P-TIASI1}
{\rm \cite{GS13}} If $f:V(G)\to \mathcal{P}(X)-\{\emptyset\}$ is a Top-IASL of a graph $G$, then $G$ must have at least one pendant vertex.
\end{prop}

\begin{prop}\label{P-TIASI2}
{\rm \cite{GS13}} Let $f:V(G)\to \mathcal{P}(X)-\{\emptyset\}$ is a Top-IASL of a graph $G$. Then, the vertices whose set-labels containing the maximal element of the ground set $X$ are pendant vertices which are adjacent to the vertex having the set-label $\{0\}$.
\end{prop}

Hence, the ground set $X$, we consider here for set-labeling the vertices of a given graph $G$ must contain $0$ as its element.

\begin{thm}\label{T-TIASI2}
{\rm \cite{GS13}} A graph $G$, on $n$ vertices, admits a Top-IASL with respect to the discrete topology of the ground set $X$ if and only if $G$ has at least $2^{|X|-1}$ pendant vertices which are adjacent to a single vertex of $G$.
\end{thm}
	
In this paper, we initiate further studies about the topological IASLs and find some new results on the topological IASL-graphs.

\section{Topological IASGL-Graphs}

Let $f$ be a topological integer additive set-indexer of a given graph $G$ with respect to a non-empty finite ground set $X$. Then, $\mathcal{T}=f(V(G))\cup \{\emptyset\}$ is a topology on $X$. Then, the graph $G$ is said to be a {\em $f$-graphical realisation} (or simply \textit{$f$-realisation}) of $\mathcal{T}$. 

In this context, the topology $\mathcal{T}$ of the ground set $X$ is said to {\em topologise} a graph $G$ with respect to an IASL $f$, if $\mathcal{T}-\{\emptyset\}=f(V(G))$. The elements of the sets $f(V)$ are called \textit{$f$-open sets} in $G$.

Is every topology of the ground set $X$ graphically realisable? The following theorem provides the solution to this problem.

\begin{thm}
Let $X$ be a non-empty finite set of non-negative integers. Then, every topology on $X$ consisting of $\{0\}$, is $f$-graphically realisable, where $f$ is an IASL defined on the graph concerned.
\end{thm}
\begin{proof}
Let $\mathcal{T}=\{A_i\subseteq X:1\le n\le r\}$ be  a topology on $X$ such that $A_1=\emptyset$, $A_2=\{0\}$ and $A_r=X$, where $r$ is a positive integer less than or equal to $|X|$. Consider the graph $G\cong K_{1,r-2}$. Label the central vertex of $G$ by the set $A_2$ and label all $r-2$ pendant vertices of $G$ by the sets $A_i:3 \le i \le r$. Therefore, $f(V)=\mathcal{T}-A_1=\mathcal{T}-\{\emptyset\}$. Therefore, $f$ is a Top-IASL on $G$. Hence, $G$ is an $f$-realisation of the topology $\mathcal{T}$ of $X$.
\end{proof}

Motivated from the studies of topological set-graceful graphs made in \cite{AGPR}, we introduce the following notion.

\begin{defn}{\rm 
An integer additive set-graceful labeling $f$ of a graph $G$, with respect to a finite set $X$, is said to be a {\em topological integer additive set-graceful labeling} (Top-IASGL) if $f(V(G))\cup\{\emptyset\}$ is a topology on $X$ and the induced edge-function $f^+(E(G))= \mathcal{P}(X)-\{\emptyset, \{0\}\}$.}
\end{defn}

\begin{thm}
For a tree, the existence of an IASGL is equivalent to the existence of an Top-IASGL.
\end{thm}
\begin{proof}
Let $G$ be a tree on $n$ vertices and $m$ edges, which admits an IASGL, with respect to a finite set $X$ of non-negative integers. Then, by Theorem \ref{T-IASGL6b}, $G\cong K_{1,l}$, where $l=2^{|X|}-2$. Then, $n=2^{|X|}-1$ and $m=2^{|X|}-2$. Therefore, $f(V(G))=\mathcal{P}(X)-\{\emptyset\}$ and $f(V(G))=\mathcal{P}(X)-\{\emptyset,\{0\}\}$. That is, $f$ is a topological IASGL of $G$. The converse part is trivial by the definition of a Top-IASGL of a graph $G$.
\end{proof}

In view of the results we have discussed so far, we can establish the following result.

\begin{thm}
If $G$ is an acyclic graph which admits a Top-IASGL, with respect to a ground set $X$, then $G$ is a star $K_{1, 2^r-2}$, where $r=2^{|X|}$. 
\end{thm}
\begin{proof}
Let $G$ be a acyclic graph which admits a Top-IASGL $f$. Since $G$ is acyclic, it is a tree. Then, as a consequence of  Theorem \ref{T-IASGL6b}, $G$ must be a star $K_{1, 2^r-2}$, where $r=2^{|X|}$.
\end{proof}

\begin{thm}\label{T-TIASGL1}
No connected $r$-regular graph $G$ admits a Top-IASGL.
\end{thm}
\begin{proof}
Let $G$ be an $r$-regular graph which admits a Top-IASGL $f$. Since $G$ is an IASG-graph, it must have at least one pendant vertex. If $v$ is a pendant vertex of $G$, then $d(v)=1$ and hence $G$ is $1$-regular. We know that $G=K_2$ is the only connected $1$-regular graph. But there exists no two element set which induces an IASGL on $K_2$. Hence, no regular graphs admit a Top-IASGL. 
\end{proof}

\begin{rem}{\rm
In view of Theorem \ref{T-TIASGL1}, no $r$-regular graph admits a Top-IASGL for $r\ge 2$. Hence, for $n\ge 2$, a cycle $C_n$ or a path $P_n$ or a complete graph $K_n$ or any complete bipartite graph $K_{n,n}$ do not admit Top-IASGLs.

Also, no complete bipartite graphs other than certain star graphs admit Top-IASGL as they do not have pendant vertices.} 
\end{rem}

During previous studies it has been established that all graphs in general do not admit Top-IASLs. It is also evident that all Top-IASLs on a given graph $G$ does not induce an integer additive set-graceful labeling on $G$. In view of these facts, the questions about the characteristics and structural properties of graphs which admit Top-IASGLs arouse much interest. 

Since the empty set $\emptyset$ can not be the set-label of any vertex of $G$, it is obvious that there exist no non-trivial graphs admit a Top-IASGL. Hence, let us proceed to check the cases involving topologies which are not indiscrete or discrete topologies of the ground set $X$. The following theorem determines a necessary and sufficient condition for the existence of an Top-IASGL for a given graph $G$.

\begin{thm}\label{T-TIASG2}
Let $X$ be a non-empty finite set. Then, a graph $G$ admits a Top-IASGL if and only if the following conditions hold.
\begin{enumerate}[itemsep=0mm]
\item[(a)] $G$ has $2^{|X|}-2$ edges and at least $2^{|X|}-(\rho+1)$ vertices, where $\rho$ is the number of subsets of $X$, which can be expressed as the sumsets of two subsets of $X$.
\item[(b)] One vertex, say $v$, of $G$ has degree $\rho''$, which is the number of subsets of X which are neither the non-trivial summands of any subsets of $X$ nor the sumsets of any elements of $\mathcal{P}(X)$.
\item[(c)] $G$ has at least $\rho'$ pendant vertices  if $X$ is not a sumset of the subsets of it and has at least $1+\rho'$ pendant vertices if $X$ is a sumset of some subsets of it, where $\rho'$ is the number of subsets of $X$, which are not the sumsets of any subsets of $X$ and not a summand of any sub set of $X$.
\end{enumerate}
\end{thm}
\begin{proof}
First assume that $G$ is a graph with the given conditions hold. Define a labeling $f$ on $G$ as explained below.

Label the vertex, say $v$, having the degree $1+2^{|X|-1}$, by the set $\{0\}$ and label an adjacent pendant vertex of $v$ by the set $X$. If $A_i\subset X$ is neither a sumset of any two subsets of $X$ nor a summand of any subset of $X$, then label one of the remaining pendant vertex of $G$ by the set $A_i$. Now, label the remaining vertices of $G$ by the remaining subsets of $X$ injectively in such way that two vertices are adjacent in $G$ if the sumset of their set-labels are subsets of $X$. Since $G$ has at least $2^{|X|}-(\rho+1)$ vertices, all the subsets, except $\{0\}$, will be the set-labels of the edges of $G$. Therefore, this labeling is a Top-IASGL of $G$.

Conversely, assume that $G$ admits a Top-IASGL $f$. Then, by Theorem \ref{T-IASGL6a}, being a IASG-graph, $G$ has a vertex, say $v$, having degree at least $\rho''$, adjacent with the vertex having set-label $\{0\}$, where $\rho''$ is the number of subsets of X which are neither the non-trivial summands of any subsets of $X$ nor the sumsets of any elements of $\mathcal{P}(X)$. 

If $A_i\subset X$ is not a sumset of of any subsets of $X$, then it must be the set-label of a vertex in $G$ that is adjacent to $v$, the vertex having $\{0\}$ as its set-label. Then, out of all $2^X-1$ subsets of $X$, the  sets which are not the sumsets of the sub sets of $X$ must be the set-labels of some vertices of $G$. Therefore, $G$ must have at least $2^{|X|}-(\rho+1)$ vertices.

Let $A_i\ne \{0\}$ be a sub set  such that it is neither a sumsets of any subsets of $X$ nor a summand of any subset of $X$, then since $A_i$ must be al element of $f^+(E(G))$, it must be the set-label of a pendant vertex that is adjacent to $v$, the vertex having $\{0\}$ as its set-label. Also, since $f(v)\cup \emptyset$ is a topology on $X$, $X\in f(V)$, irrespective of whether it is a sumset of its sumset of its subsets or not. Then, by Theorem \ref{T-IASGL3}, $X$ must also be the set-label of one pendant vertex that is adjacent to $v$. Therefore, the minimum number of pendant vertices that are adjacent to $v$ must be $1+\rho'$ if $X$ is not a sumset of its subsets and $\rho'$ if $X$ is a sumset of its subsets.
\end{proof}

Figure \ref{G-Top-IASGL} depicts a topological integer additive set-graceful labeling of a graph.

\begin{figure}[h!]
\centering
\includegraphics[width=0.75\linewidth]{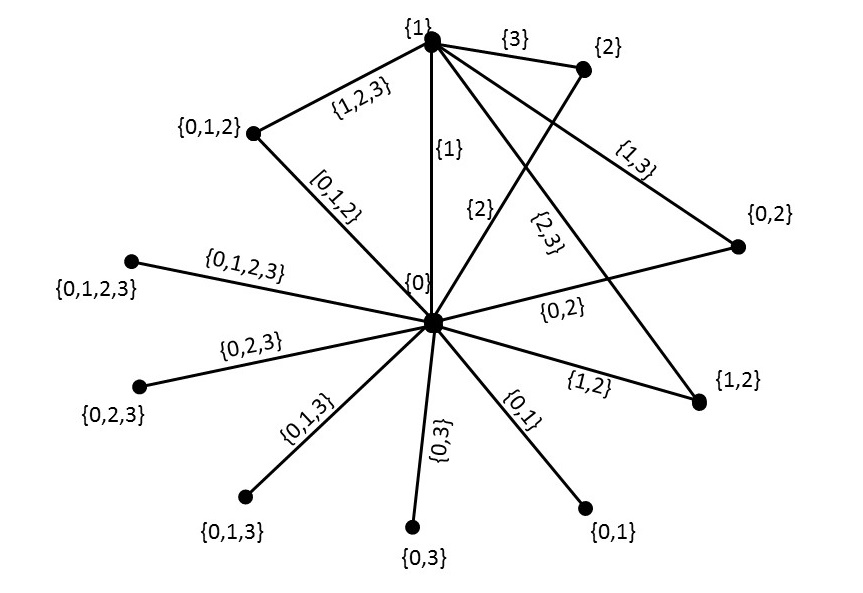}
\caption{A graph with a topological IASGL.}
\label{G-Top-IASGL}
\end{figure}

In this context, the following theorem discusses the admissibility of a Top-IASGL with respect to the discrete topology of the ground set $X$.

\begin{thm}\label{T-TIASG3}
A graph $G$ admits a Top-IASGL $f$ with respect to the discrete topology of a non-empty finite set $X$ if and only if $G\cong K_{1,2^{|X|}-2}$.
\end{thm}
\begin{proof}
If $G\cong K_{1,2^{|X|}-2}$, then by Theorem \ref{T-TIASGL1}, $G$ is an $f$-realisation of the discrete topology of $X$. Then, $f$ is a topological IASGL of $G$.

Conversely, assume that $G$ admits a Top-IASGL, say $f$, with respect to the discrete topology of a non-empty finite set $X$. Therefore, $|f(V(G))|=2^{|X|}-1$ and $|f^+(E(G))|=2^{|X|}-2$. That is, $|E(G)|=|V(G)|-1$. Therefore, $G$ is a tree on $2^{|X|}-1$ vertices. Then, by Theorem \ref{T-IASGL6b}, $G\cong K_{1,2^{|X|}-2}$.
\end{proof}

Figure \ref{G-Top-IASGLa} depicts a topological integer additive set-graceful labeling of a graph with respect to the discrete topology of the ground set $X$.

\begin{figure}[h!]
\centering
\includegraphics[width=0.65\linewidth]{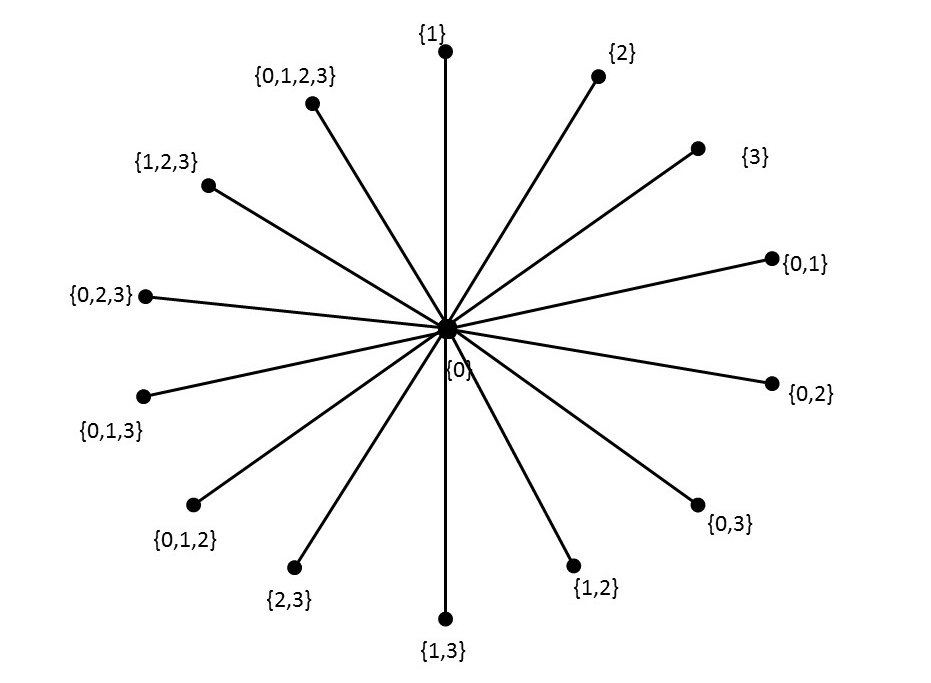}
\caption{A star graph with a topological IASGL.}
\label{G-Top-IASGLa}
\end{figure}

\section{Conclusion}
In this paper, we have discussed about certain the properties and characteristics of topological IASL-graphs. More properties and characteristics of TIASLs, both uniform and non-uniform, are yet to be investigated. The following are some problems which demand further investigation.

\begin{prob}{\rm
Characterise the graphs which are Top-IASS-graphs but not Top-IASG-graphs.}
\end{prob}

\begin{prob}{\rm
Find the smallest ground set $X$, with respect to which a given graph admits a Top-IASGL.}
\end{prob}

\begin{prob}{\rm
Find the smallest ground set $X$, with respect to which a given graph admits a Top-IASSL.}
\end{prob}

The problems of establishing the necessary and sufficient conditions for various graphs and graph classes to have certain other types IASLs are also still open. Studies about those IASLs which assign sets having specific properties, to the elements of a given graph are also noteworthy. All these facts highlight a wide scope for further studies in this area.

\end{document}